\documentclass{amsart}
\usepackage{amsmath, amsthm, amssymb, amsfonts}
\usepackage{bm}
\usepackage{dsfont}
\usepackage[utf8]{inputenc}
\usepackage{rotate}
\usepackage{tikz}
\usepackage{tikz-cd}
\usepackage[arrow,matrix,curve]{xy} 	
\usepackage{xcolor}
\usepackage{todonotes}

\theoremstyle{plain}
\newtheorem{theorem}{Theorem}[section]
\newtheorem{proposition}[theorem]{Proposition}
\newtheorem{lemma}[theorem]{Lemma}
\newtheorem{corollary}[theorem]{Corollary}

\theoremstyle{definition}

\newtheorem{remark}[theorem]{Remark}
\newtheorem{example}[theorem]{Example}

\newcommand{\norm}[1]{\left\lVert#1\right\rVert}

\newcommand{\vol}{{\rm{vol}}}

\newcommand{\RR}{{\mathbb{R}}}

\newcommand{\NN}{{\mathbb{N}}}


\title[Discrepancy, Finite Gap Properties, and Pair Correlations]{Some Connections Between Discrepancy, Finite Gap Properties, and Pair Correlations}
\date{\today}

\author{Christian Wei\ss{}}
\address{{\bf{Ruhr West University of Applied Sciences,}}\\ {{Department of Natural Sciences, Duisburger Str. 100,}}\\{{45479 M\"ulheim an der Ruhr, Germany}}}
\email{christian.weiss@hs-ruhrwest.de}

\begin{document}

\maketitle

\begin{abstract}  A generic uniformly distributed sequence $(x_n)_{n \in \mathbb{N}}$ in $[0,1)$ possesses Poissonian pair correlations (PPC). Vice versa, it has been proven that a sequence with PPC is uniformly distributed. Grepstad and Larcher gave an explicit upper bound for the discrepancy of a sequence given that it has PPC. As a first result, we generalize here their result to the case of $\alpha$-pair correlations with $0 < \alpha < 1$. Since the highest possible level of uniformity is achieved by low-discrepancy sequences it is tempting to assume that there are examples of such sequences which also have PPC. Although there are no such known examples, we prove that every low-discrepancy sequence has at least $\alpha$-pair correlations for  $0 < \alpha <1$. According to Larcher and Stockinger, the reason why many known classes of low-discrepancy sequences fail to have PPC is their finite gap property. In this article, we furthermore show that the discrepancy of a sequence with the finite gap property plus a condition on the distribution of the different gap lengths can be estimated. As a concrete application of this estimation, we re-prove the fact that van der Corput and Kronecker sequences are low-discrepancy sequences. Consequently, it follows from the finite gap property that these sequences have $\alpha$-pair correlations for $0 < \alpha < 1$.
\end{abstract}

\section{Introduction}

Let $\varphi = \frac{1+\sqrt{5}}{2}$ be the golden mean and consider the Kronecker sequence $(x_n)_{n \in \mathbb{N}} = (\left\{n\varphi\right\})_{n \in \mathbb{N}}$, where $\left\{ x \right\} := x - \lfloor x \rfloor$ denotes the fractional part of $x \in \mathbb{R}$. It exhibits several remarkable properties. First, it is a uniformly distributed sequence and even a classical example of a low-discrepancy sequence which means that it is roughly speaking as uniformly distributed as possible, see e.g. \cite{Nie92}. Second, the famous Three Gap Theorem going back to S\'os in \cite{Sos58} holds: if we place the first $N$ points of the sequence on a circle, then there are at most three distinct distances between adjacent points. Despite this high level of uniformity, $(x_n)_{n \in \mathbb{N}}$ fails to have another property of a generic uniformly distributed sequence: the Kronecker sequence of the golden mean does not have Poissonian pair correlations, \cite{LS20}. On the contrary, it is known that $(x_n)_{n \in \mathbb{N}}$ very tightly fails to have this property because it has $\alpha$-pair correlations for all  $0 < \alpha < 1$, see \cite{WS19}.\\[12pt]
This paper is intended to shed more light on the connections between the properties mentioned in the first paragraph. Let us start with the third-named concept which was originally introduced in \cite{RS98}. A sequence $(x_n)_{n \in \mathbb{N}}$ in $[0,1)$ has Poissonian pair correlations if the pair correlation statistics
\begin{align} \label{eq5}
F_N(s) := \frac{1}{N} \# \left\{ 1 \leq l \neq m \leq N \ : \ \norm{x_l - x_m} \leq \frac{s}{N} \right\}
\end{align}
tends to $2s$ for all $s \in \mathbb{R}_{\geq 0}$ as $N \to \infty$, where $\norm{\cdot}$ is the distance to the nearest integer. Although Poissonian pair correlations are a generic property of uniformly distributed sequences in $[0,1)$, only some explicit example sequences with this property have been found hitherto, see \cite{BMV15} and \cite{LST21}. \\[12pt] 
The concept of Poissonian pair correlations can be generalized to dimension $d$ and an arbitrary norm $\norm{\cdot}$ on the $d$-dimensional torus $\mathds{T}_d = [0,1)^d$ by saying that a sequence $(x_n)_{n \in \mathbb{N}}$ in $\mathds{T}_d$ has Poissonian pair correlations with respect to $\norm{\cdot}$ if 
$$\lim_{N \to \infty} \frac{1}{N^2} \frac{\# \left\{ 1 \leq l \neq m \leq N \ : \ \norm{x_l - x_m} \leq \frac{s}{N^{1/d}} \right\}}{\vol(B(0,sN^{-1/d}))} = 1$$
for all $s\geq 0$, where $\vol(B(0,r))$ is the (Lebesgue) volume of the ball of radius $r$, compare \cite{NP07}. Note that $\vol(B(0,sN^{-1}) = 2s/N$ for e.g. the maximum metric and thus the definition truly generalizes the one-dimensional case from \eqref{eq5}. As a concrete example in higher dimension, the \textbf{following} was studied in detail in \cite{HKL19}: let the norm of a point $x \in \RR$ be again defined by its distance to the nearest integer  and for $(x_1,\ldots,x_d) \in \mathbb{R}^d$ set
$$\norm{x}_\infty :=  \max(\norm{x_1},\ldots,\norm{x_d}).$$
According to \cite{HKL19}, Theorem~1, uniformly distributed sequences then generically have Poissonian pair correlations. Amongst others, the authors moreover showed that also the opposite is true, i.e. having Poissonian pair correlations implies uniform distribution of the sequence (Theorem 2 in \cite{HKL19}). The choice of the exponent $\tfrac{1}{d}$ in the definition of Poissonian pair correlations is a coherent approach because if the exponent was $>\tfrac{1}{d}$, then obviously the left hand side would go to $\infty$ and thus $\alpha = \tfrac{1}{d}$ is the biggest possible choice. However, the exponent might be chosen smaller than $\tfrac{1}{d}$ which leads for $0 < \alpha < \tfrac{1}{d}$ to the expression
$$F_N^{\alpha}(s) :=  \frac{1}{N^2} \frac{\# \left\{ 1 \leq l \neq m \leq N \ : \ \norm{x_l - x_m} \leq \frac{s}{N^{\alpha}} \right\}}{\vol(B(0,sN^{-\alpha}))}.$$
If $F_N^\alpha(s) \to 1$ for all $s \geq 0$, then a sequence is said to have $\alpha$-pair correlations, see \cite{NP07}. In other contexts $\alpha$-pair correlations are also called number variance (with exponent $\alpha$), compare \cite{Mar07}.  In \cite{HZ21}, Theorem 4, it is proven for dimension $d=1$ and $\norm{\cdot}_{\infty}$ that $F_N^{\alpha_1}(s) \to 1$ implies $F_N^{\alpha_2}(s) \to 1$ for all $s \geq 0$ if $\alpha_1 \geq \alpha_2$. Hence the larger $\alpha$ is, the harder it is to achieve $\alpha$-pair correlations under the mentioned conditions. The proof from \cite{HZ21} is however general enough to almost verbatim transfer the same implication also to our more general setting. \\[12pt]
Next recall that the discrepancy of a sequence $(x_n) \in [0,1)^d$ is defined by
$$D_N(x_n) := \sup_{B \subset [0,1)^d} \left| \frac{1}{N}\#(\left\{x_i | 1 \leq i \leq N \right\} \cap B) 
%\frac{\sum_{i=1}^N \mathds{1}_B(p_n) }{N} 
- \lambda_d(B) \right|,$$
where the supremum is taken over all intervals $B = [a,b) \subset [0,1)^d$. If the $\sup$ is further restricted to sets of the form $B^* = [0,b) \subset [0,1)^d$, then we speak of the star-discrepancy $D_N^*(x_n)$. In fact, both types of discrepancies are related by the inequalities $D_N^*(x_n) \leq D_N(x_n) \leq 2^dD_N^*(x_n)$, which means that their asymptotic behavior is equivalent, see e.g. \cite{KN74}. Therefore, it often depends on the context if it is more convenient to work with the usual discrepancy or the star-discrepancy. If the condition on the $\sup$ is instead relaxed to all convex subsets $C \subset [0,1)^d$, then the corresponding quantity $J_N(x_n)$ is called isotropic discrepancy and the inequalities $D_N(x_n) \leq J_N(x_n) \leq 4dD_N(x_n)^{1/d}$ hold. For more details we refer the reader to \cite{DP10},  \cite{Nie92}. If a sequence $(x_n)_{n \in \mathbb{N}} \in [0,1)^d$ exhibits a star-discrepancy of order
\begin{equation}\label{eq:asymp_bound}
D_N^*(x_n) = O(N^{-1}(\log N)^{d}),
\end{equation}
then it is called a low-discrepancy sequence. It is conjectured that this is the fastest possible rate of convergence. In fact, this is known to be true for dimension one by the work of Schmidt, \cite{Sch72}. Since low-discrepancy sequences may be interpreted as sequences which are as uniformly distributed as possible, there might be examples of sequences in this class having Poissonian pair correlations. However, all attempts to find such examples have failed so far and it has even been proved for many explicit types of low-discrepancy sequences (in dimension $d = 1$ and also in higher dimensions) that they do not have Poissonian pair correlations, see \cite{BCC19}, \cite{LS20}, \cite{PS19}, \cite{WS19}. In this article, we will argue why it might nonetheless be worth to keep on looking for examples of Poissonian pair correlations in the class of low-discrepancy sequences. We will not only restrict our analysis to dimension $1$ but also consider higher dimensions. It is the first aim of this article to show that \textit{all} low-discrepancy have $\alpha$-pair correlations for $\norm{\cdot}_\infty$ and $\tfrac{1}{d} > \alpha > 0$. 
\begin{theorem} \label{main:thm} Let $(x_n)_{n \in \mathbb{N}} \in \mathds{T}_d$ (equipped with $\norm{\cdot}_\infty$) be a sequence with $D_N(x_n) = o(N^{-(1-\varepsilon)})$ for $0 < \varepsilon < \tfrac{1}{d}$, then $(x_n)_{n \in \mathbb{N}}$ has $\alpha$-pair correlations for any  $0<\alpha<(\tfrac{1-\varepsilon}{d})$.
\end{theorem}
As an immediate consequence we can derive that low-discrepancy sequences \textit{almost} have Poissonian pair correlations. 
\begin{corollary} \label{cor:ld:ppc} Let $(x_n) \in \mathds{T}_d$ (equipped with $\norm{\cdot}_\infty$) be a low-discrepancy sequence. Then $(x_n)_{n \in \mathbb{N}}$ has $\alpha$-pair correlations for all $0 < \alpha < \tfrac{1}{d}$.
\end{corollary}
Using technical number theoretic arguments, the corresponding statement in $d=1$ was proved for the Kronecker sequence $\left\{ n \varphi \right\}$ and van der Corput sequences in \cite{WS19}. In \cite{TW20}, the authors prove that Kronecker sequences $(\left\{nz\right\})_{n \in \mathbb{N}}$ where the partial quotients in the continued fraction expansion of $z$ satisfy a certain growth condition possess this property, for details see Remark~\ref{rem:tw20}. Our proof of Theorem~\ref{main:thm} is short and only relies on the order of convergence of the star-discrepancy although there are similarities to the proof in \cite{TW20}.\\[12pt]
The applications of Theorem~\ref{main:thm} are not limited to low-discrepancy sequences but can be used to show that also higher-dimensional Kronecker sequences often have $\alpha$-pair correlations (it is not known whether they are low-discrepancy sequences, see \cite{Nie92}) although they fail to have Poissonian pair correlations, cf. \cite{HKL19}, Theorem~3. 
\begin{corollary} \label{cor:kronecker_mult} Let $z:=(z_1,\ldots,z_d) \in \mathbb{R}^d$ such that $1,z_1,\ldots,z_d$ are linearly independent over $\mathbb{Q}$. Then $\left\{nz\right\}$ has $\alpha$-pair correlations for all $0 < \alpha < \frac{1}{d}$ and $\norm{\cdot}_\infty$. 
\end{corollary}
Indeed, Corollary~\ref{cor:kronecker_mult} follows from a result in \cite{Nie72} according to which multi-dimensional Kronecker sequences satisfying the mentioned condition have discrepancy of order $O(N^{-(1-\varepsilon}))$ for all $\varepsilon > 0$.\\[12pt]
Conversely, it was independently proved in \cite{ALP18} and \cite{GL17} that a sequence which has Poissonian pair correlations is also uniformly distributed. Alternative proofs of this fact were also given in \cite{Ste17} and \cite{Mar20}. In this paper, we generalize the quantitative version from \cite{GL17} to $\alpha$-pair correlations. 
\begin{theorem} \label{thm:paircor_discrepancy} Let $(x_n)_{n \in \mathbb{N}}$ be a sequence in $[0,1)$, and suppose that there exists a function $F: \mathbb{N}\times\mathbb{N} \to \mathbb{R}$ which is monotonically increasing in its first argument, and which satisfies
$$\max_{s=1,\ldots,K} \left| \frac{1}{2s} \# \left\{ 1 \leq l \neq m \leq N : \norm{x_l-x_m} < \frac{s}{N^\alpha} \right\} - N^{2-\alpha}\right| < F(K,N)$$
for some $0 < \alpha \leq 1$ and all $K \leq N/2$. Then there exists an integer $N_0 > 0$ such that for all $N \in \mathbb{N}, N \geq N_0$, and arbitrary $K$ satisfying
%$$\min \left( \frac{1}{2} N^{\tfrac{2}{5}\alpha}, \frac{N^\alpha}{F(K^2,N)}\right) \leq K \leq N^{\tfrac{2}{5}\alpha},$$
$$\frac{1}{2} N^{\tfrac{2}{5}\alpha} \leq K \leq N^{\tfrac{2}{5}\alpha}.$$
we have 
$$N D_N^*(x_n) \leq 5 \max\left(N^{1-\tfrac{1}{5}\alpha},\sqrt{N^\alpha \cdot F(K^2,N)} \right).$$
\end{theorem}
In comparison to the result in \cite{GL17}, our condition on $K$ is a little bit stronger but in exchange the theorem can be formulated more compactly. From this result we can immediately deduce the following corollary in a very similar manner as in \cite{GL17}.
\begin{corollary} \label{cor:alpha-pair} If the sequence $(x_n)_{n \in \mathbb{N}}$ in $[0,1)$ has $\alpha$-pair correlations for an $0 < \alpha < 1$, then it is uniformly distributed.
\end{corollary}
The corollary was proven for Poissonian pair correlations in \cite{ALP18} independent of \cite{GL17} by a different type of argument. Also the statement for $\alpha$-pair correlations is well-known due to Steinerberger in \cite{Ste20}. An alternative proof was moreover given in \cite{Coh21}.\\[12pt]
It is enlightening to compare Corollary~\ref{cor:kronecker_mult} in dimension $d=1$ more closely to a result from \cite{LS20} which traces back the non-Poissonian pair correlations of Kronecker sequences to their finite gap property. Before we come to it let us at first recall the so-called three gap theorem. We formulate it here similarly as in \cite{AB98} (without the point $0$) and for that purpose denote the continued fraction expansion of $z \in \mathbb{R} \setminus \mathbb{Q}$ by $[a_0,a_1,a_2,\ldots]$ and let the convergents be $p_n/q_n$ (see Section~\ref{sec:disc_pair}). %We formulate it here in terms of the Ostrowski expansion, as it is implicitly done in \cite{PSZ16}, see also \cite{Wei20} for a slightly different formulation. For that purpose let $z \in \mathbb{R} \setminus \mathbb{Q}$ have continued fraction expansion $[a_0,a_1,a_2,\ldots]$ with convergents $p_n/q_n$ (see Section~\ref{sec:disc_pair}). Then $N \in \mathbb{N}$ can be uniquely written as
%$$N  = \sum_{n=1}^{l(N)} b_n q_n$$
%with $0 \leq b_n \leq a_n$ and $b_{n-1} = 0$ if $b_n = a_n$  (Ostrowski expansion).
\begin{theorem}[Three Gap Theorem] \label{thm:3gap} Let $(nz)_{n \in \mathbb{N}}$ be the Kronecker sequence of $z \in \mathbb{R} \setminus \mathbb{Q}$ and write $N \in \mathbb{N}$ uniquely as
	$$N = c q_n + q_{n-1} + r$$
with $1 \leq c \leq a_{n+1}$ and $0 \leq r < q_n$. Then the gaps between two adjacent terms in the set $\left\{ \left\{nz\right\} \, : \, 1 \leq n \leq N \right\}$ that can appear have lengths
\begin{align*}
L_1 & = \norm{q_n z },\\
L_2 & = \norm{q_{n-1} z } - cL_1,\\
L_3 & = L_1 + L_2,
\end{align*}
and their multiplicities are
\begin{align*}
N_1 & = N - q_n,\\
N_2 & = r,\\
N_3 & = q_n - r.
\end{align*}
\end{theorem}
The following theorem by Larcher and Stockinger in \cite{LS20} indeed shows that the three gap property of Kronecker sequences prevents them from having Poissonian pair correlations.
\begin{theorem}[Larcher, Stockinger, \cite{LS20}, Theorem 1] \label{thm:LS} Let $(x_n)_{n \in \mathbb{N}}$ be a sequence in $[0,1)$ with the following property: There is an $s \in \mathbb{N}$, positive real numbers $K,\gamma \in \mathbb{R}$, and infinitely many $N \in \mathbb{N}$ such that the point set $x_1,\ldots,x_N$ has a subset with $M \geq \gamma N$ elements, denoted by $x_{j_1},\ldots,x_{j_M}$, which are contained in a set of points with cardinality at most $KN$ having at most $s$ different distances between neighbouring sequence elements, so-called gaps. Then $(x_n)_{n \in \mathbb{N}}$ does not have Poissonian pair correlations.
\end{theorem}
A main part of the proof of Theorem~\ref{thm:LS} is contained in Proposition~1 of \cite{LS20}. Therein a trichotomy for gap lengths is identified. This approach, can be transferred from the context of Poissonian pair correlations to $\alpha$ pair correlations. Suppose there is a a sequence $N_1 < N_2 < ...$ such that the set of points $x_1,\ldots,x_{N_i}$ has exactly $s$ different gap lengths (due to the Three Gap Theorem~\ref{thm:3gap} we have $s \in \left\{ 2, 3\right\}$ for Kronecker sequences). For every $i$ denote the lengths of these gaps by $L_1^{(i)} < L_2^{(i)} < \ldots < L_s^{(i)}$. Let $1 \geq \alpha > 0$. Then there exist $w_1 = w_1(\alpha) < w_2 = w_2(\alpha) \in \mathbb{N}$ and $K_1 = K_1(\alpha), K_2 = K_2(\alpha) > 0$, which lead to the following definitions:
\begin{itemize}
    \item For \textbf{$\alpha$-small gap lengths} we have $\lim_{i \to \infty} N_i^\alpha L_j^{(i)} = 0$ for $j=1,\ldots,w_1$.
    \item For \textbf{$\alpha$-intermediate gap lengths} with $j=w_1+1,\ldots,w_2-1$, the inequalities 
    $$K_1/N_i^\alpha \leq L_j^{(i)} \leq K_2/N_i^\alpha$$
    hold for all $i \in \mathbb{N}$.
    \item For \textbf{$\alpha$-large gap lengths} with $j=w_2,\ldots,s$ we have
    $\lim_{i \to \infty} N_i^\alpha L_j^{(i)} = \infty.$
\end{itemize}
It is obvious that for any $1 > \alpha>0$ Kronecker sequences never possess $\alpha$-large gaps because for $N = q_{n+1}$ we have
$$N^\alpha L_{\max} \leq q_{n+1}^\alpha \norm{q_{n} z} \leq q_{n+1}^{\alpha-1} \to 0.$$
More generally, for $\alpha_1 > \alpha_2 > 0$ being an $\alpha_2$-large ($\alpha_2$-intermediate) gap implies that the gap is also a $\alpha_1$-large (at least $\alpha_2$-intermediate) gap. On the other hand, being a $1$-large gap implies that the gap length is $\geq \tfrac{1}{N}$ for $N$ large enough. Conversely, $1$-intermediate or $1$-large gaps thus exist for all sequences with the finite gap property. The following two obstructions for Poissonian pair correlations have been identified in Proposition~1 of \cite{LS20} and the proof therein verbatim works for $\alpha$ pair correlations: 
\begin{itemize}
    \item \textbf{Obstruction 1:} There exists an $\alpha$-intermediate gap length.
    \item \textbf{Obstruction 2:} The largest $\alpha$-small gap is also a $1$-small gap (in other words, $d_{w_{1}^{(i)}}N_i \to 0$).
\end{itemize}
Note that Kronecker sequences fail to have Poissonian pair correlations because Obstruction 1 holds for $\alpha=1$. In contrast, we prove (independently of Theorem~\ref{main:thm}) that a Kronecker sequence cannot fulfill Obstruction 1 or Obstruction 2 if $\alpha < 1$ for $z$ algebraic of degree $\leq 2$.
\begin{proposition} \label{prop:obstructions} For all algebraic $z \in \mathbb{R}$ of degree $\leq 2$, the Kronecker sequence does neither fulfill Obstruction 1 nor Obstruction 2 for $\alpha < 1$.
\end{proposition}
A corresponding property as assumed in Theorem~\ref{thm:LS} also holds for Kronecker sequences in dimensions $d \geq 2$, compare e.g. the recent publications \cite{HM20}, \cite{HR20}: the number of nearest neighbor distances of multi-dimensional Kronecker sequences is universally bounded for the Euclidean and the maximum metric. Thus, it is an obvious question to ask whether it is true also in higher dimensions that a finite number of nearest neighbor distances prevents a sequence from having Poissonian pair correlations. This question is left open for future research.\\[12pt]
Conversely, we show here that if a sequence only has a finite number of gap lengths and if these different gap lengths distribute over the unit interval in a nice way, then the star-discrepancy cannot be too large. In view of the observations in \cite{LS20} (finite gap property implies that the sequence does not have Poissonian pair correlations) and Theorem~\ref{main:thm} (low-discrepancy \textit{almost} implies Poissonian pair correlations), this result might be a bit surprising at first sight but completes the picture in a sense.
\begin{theorem} \label{thm:gap_discrepancy} Let $(x_n)_{n \in \mathbb{N}}$ be an arbitrary sequence and denote by $(x_n^*)_{n=1}^N$ the sequence of the first $N \in \mathbb{N}$ elements of $(x_n)_{n \in \mathbb{N}}$ ordered by size. Assume that the gap lengths of $(x_n^*)_{n=1}^N$ are $L_1<L_2<\ldots<L_K$ and that for the largest gap length $L_K = \frac{R+2}{N}$ for some $R > - 1$ holds. Furthermore let $N_1,N_2,\ldots,N_k$ denote the corresponding multiplicities of the occurring gap lengths. If for all $j=1,\ldots,N$ we have
$$x_j^* = x_1^* + \sum_{k=1}^K n_k(j) L_k$$
with $n_k(j) \in [ \tfrac{N_k}{N}j - \varepsilon; \tfrac{N_k}{N}j + \varepsilon]$ for $k=1,\ldots,K$, then
$$D_N^*(x_1,\ldots,x_N) \leq \frac{R+3}{N} + \varepsilon \sum_{k=1}^K L_k.$$
\end{theorem}
\begin{remark} The assumption $L_K = \frac{R+2}{N}$ is not a restriction for fixed $N$ because we always have $NL_K \geq 1$ and thus $L_K \geq \frac{1}{N}$. The condition that the largest gap length converges to zero (and so do all other gap lengths) means that $R=R(N)$ is $o(N)$. The latter is obviously a necessary condition for the sequence to be uniformly distributed. Moreover, note that also $K=K(N), L_j(N)$ and $\varepsilon=\varepsilon(N)$ depend on $N$.
\end{remark}
Theorem~\ref{thm:gap_discrepancy} may be regarded as a tool to bound the discrepancy and even to prove for certain sequences that they have low-discrepancy. In fact, we discuss in Section~\ref{sec:Gaps_and_Discrepancy}, that Theorem~\ref{thm:gap_discrepancy} can indeed be applied to show for both Kronecker and van der Corput sequences that they are low-discrepancy sequences. This is done in Example~\ref{exa:kronecker} and Example~\ref{exa:vdc} respectively by using the fact that these two types of sequences have the finite gap property, i.e. $K=O(1)$ for $N \to \infty$, and proving that $\varepsilon=\varepsilon(N)=O(\log(N))$. These examples indicate that looking at the gap structure of a sequence in a precise way is sufficient for calculating the asymptotic behavior of the discrepancy. Since the machinery is kept quite universal, it can presumably be applied to other classes of sequences, too. For example, also LS-sequences from \cite{Car12} and the low-discrepancy sequences stemming from interval exchange transformations in \cite{Wei19} have a finite gap property.

\section{Discrepancy and Pair Correlations} \label{sec:disc_pair}
A sequence $(x_n) \in [0,1)^d$ is uniformly distributed if and only if $\lim_{N \to \infty} D_N(x_n) \to 0$. According to Theorem~\ref{main:thm} the speed of convergence of the discrepancy to $0$ determines the $\alpha$-pair correlation statistics. Vice versa Theorem~\ref{thm:paircor_discrepancy} implies that the relation is to a certain extent mutual. This section is dedicated to the proof of these two theorems and some related results.
\begin{proof}[Proof of Theorem~\ref{main:thm}]
Assume that $(x_n)_{n \in \mathbb{N}}$ is a sequence with $D_N(x_n) = o(N^{-(1-\varepsilon)})$ for $0 < \varepsilon < \tfrac{1}{d}$ and let $s > 0$ be arbitrary. A ball of radius $sN^{-\alpha}$ in $\norm{\cdot}_\infty$ contains 
$$N\vol(B(0,sN^{-\alpha}) + o(N^{\varepsilon})$$ points because of the order of convergence of $D_N(x_n)$. In other words, for an arbitrary $x_m$ with $1 \leq m \leq N$ there exist $N\vol(B(0,s/N^{-\alpha}) + o(N^{\varepsilon})$ points $x_l$ with $\norm{x_l-x_m}_\infty \leq sN^{-\alpha}$, because of the compatibility of the usual discrepancy with $\norm{\cdot}_\infty$. Keeping in mind that $\vol(B(0,sN^{-\alpha})) = o(N^{-d\alpha})$, we obtain 
\begin{align*}
	F_N^\alpha(s) = \frac{1}{N^{2}} & \frac{\# \left\{ 1 \leq l \neq m \leq N \ : \ \norm{x_l - x_m}_\infty \leq \frac{s}{N^{\alpha}} \right\}}{\vol(B(0,sN^{-\alpha}))}\\
	& = \frac{1}{N^2} \cdot N \cdot \frac{N\vol(B(0,sN^{-\alpha}) + o(N^{\varepsilon})}{\vol(B(0,sN^{-\alpha}))}\\
	& = 1 + o(N^{\varepsilon+\alpha \cdot d - 1})
\end{align*}
Letting $N \to \infty$ completes the proof Theorem~\ref{main:thm} because $\alpha < \tfrac{1-\varepsilon}{d}$.
\end{proof}
For arbitrary norms $\norm{\cdot}$ we can, in general, only use the isotropic discrepancy $J_N(x_n)$ instead of the discrepancy because the former takes into account all convex sets and hence can be applied to sets of the form $\norm{x_l-\cdot}$ for $l \leq N$. From $D_N(x_n) \leq J_N(x_n) \leq 4dD_N(x_n)^{1/d}$, the following corollary can be immediately derived.
\begin{corollary} Let $(x_n) \in \mathds{T}_d$ (equipped with an arbitrary norm $\norm{\cdot}$)) be a low-discrepancy sequence. Then $(x_n)_{n \in \mathbb{N}}$ has $\alpha$-pair correlations for all $0 < \alpha < \tfrac{1}{d^2}$.
\end{corollary}
We now briefly fix notation and summarize some of the important properties of continued fractions. For more details, we refer the reader to \cite{BS96,Nie92}. Let $[a_0;a_1,\ldots]$ be the continued fraction expansion of $z$ and denote the corresponding sequence of convergents by $(p_n/q_n)_{n \in \mathbb{N}_0}$. Recall that
\begin{align*} %\label{eq:p}
p_{-2} = 0, p_{-1} = 1, p_n = a_np_{n-1} + p_{n-1}, n \geq 0
\end{align*}
\begin{align*} %\label{eq:q}
q_{-2} = 1, q_{-1} = 0, q_n = a_nq_{n-1} + q_{n-1}, n \geq 0
\end{align*}
\begin{remark} \label{rem:tw20} In \cite{TW20}, the assertion of Corollary~\ref{cor:ld:ppc} was shown for Kronecker sequences $(\{nz\})_{n \in \mathbb{N}}$ under the following growth condition on the partial quotients $a_i$ of $z$: Given $N \in \mathbb{N}$, let $i(N)$ be such that the convergent denominator $q_{i(N)}$ of $z$ satisfies $q_{i(N)} \leq N < q_{i(N)+1}$. If for each $\varepsilon > 0$, we have
	$$\sum_{j \leq i(N)} a_j \ll N^\varepsilon,$$
then the Kronecker sequence has $\alpha$-pair correlations for all $0 < \alpha < 1$. 
\end{remark}
In order to show Proposition~\ref{prop:obstructions}, we recall that algebraic numbers of degree $\leq 2$ have bounded continued fraction expansion and prove the following lemma.
\begin{lemma} \label{lem:intermediate_gaps} If the continued fraction expansion of $z$ has bounded coefficients, then for any $\alpha < 1$, the Kronecker sequence $(x_n)_{n \in \mathbb{N}} = (\left\{nz \right\})_{n \in \mathbb{N}}$ only has $\alpha$-small gap lengths. If we restrict to the subsequence $(\left\{ q_i z\right\})_{i \in \mathbb{N}}$ of the Kronecker sequence, where only two gap lengths appear, then the larger gap length is a $1$-intermediate gap (independent of whether the $a_i$ are bounded). 
\end{lemma}
\begin{proof} Let $q_{n} < N \leq q_{n+1}$. According to the Three Gap Theorem~\ref{thm:3gap} the gap length $\norm{q_n z}$ is the minimal one and the maximal possible one is $\norm{q_{n-1} z } + \norm{q_n z}$. 
From the fact that the $a_i$ are bounded, say by $K \in \mathbb{N}$, we obtain
$$q_{n} < q_{n+1} < (K+1)q_n$$
for all $n \in \mathbb{N}$. Using the well-known inequalities 
$$\frac{1}{q_{n+1}+q_n} \leq \norm{q_n z} \leq \frac{1}{q_{n+1}}$$
yields
$$N^\alpha L_{\max} \leq 2 N^\alpha \norm{q_{n-1}z} \leq 2\frac{q_{n+1}^\alpha}{q_n} \leq 2(K+1)^\alpha \frac{1}{q_n^ {1-\alpha}}$$
Therefore $N^\alpha \cdot L_i \to 0$ for all $\alpha < 1$ and $L_i$ any appearing gap length at step $N$. In our wording this means that only $\alpha$-small gap lengths appear for $\alpha < 1$. However, $N \cdot L_{\max} \to 0$ does not hold because for $N=q_{n+1}$ we have
$$\frac{1}{2} < \frac{q_{n+1}}{q_{n+1}+q_{n}} \leq q_{n+1} (\norm{ q_nz} + \norm{q_{n-1}z}) < 2$$
and thus a $1$-intermediate gap length occurs.
\end{proof}
Finally, we come to the proof of Theorem~\ref{thm:paircor_discrepancy}. Since the arguments in our proof are essentially the same as in \cite{GL17}, Theorem 1, we leave away the derivation of two inequalities here and refer the reader to the respective article for more details.
\begin{proof}[Proof of Theorem~\ref{thm:paircor_discrepancy}]
At first we set 
$$H(N,K) := 5 \cdot \max \left( N^{1-\frac{1}{5}\alpha},\sqrt{N^\alpha \cdot F(K^2,N)}\right).$$
and assume that $N D_N^* > H(N,K)$ for infinitely many $N,K$ which will lead to a contradiction at the end of the proof. This implies that there exist sequences $1 < N_1 < N_2 < \ldots$ and $(K_j)_{j \in \mathbb{N}}$ of integers and a sequence of real numbers $(B_j)_{j \in \mathbb{N}}$ such that we have without loss of generality
$$\# \left\{ 1 \leq n \leq N_j : x_n \in [0,B_j) \right\} - N_jB_j > H(N_j,K_j)$$
for all $j$ (the other possible case can be treated similarly). Let $N:=N_j, K:=K_j, B:=B_j$ and $H:=H(N_j,K_j)$. Then the equation above implies
$$N - NB - H > 0.$$
Furthermore we define the numbers 
$$A_i:= \# \left\{ 1 \leq n \leq N : x_n \in \left[ i \cdot \frac{K}{N^\alpha},(i+1) \cdot \frac{K}{N^\alpha}\right) \right\}$$
for $i=0,1,\ldots,\lfloor N^\alpha/K \rfloor - 1$ and
$$A_{\lfloor N^\alpha / K \rfloor}:= \# \left\{ 1 \leq n \leq N : x_n \in \left[ \left\lfloor \frac{N^\alpha}{K} \right\rfloor \cdot \frac{K}{N^\alpha},1\right) \right\}.$$
For arbitrary $l \in \mathbb{N}$ we set $A_l:=A_{l \mod \lfloor N^\alpha/K\rfloor}$. Finally, we introduce the notation
$$\mathcal{H}_L:= \# \left\{1 \leq l \neq m \leq N: \norm{x_l-x_m} \leq \frac{KL}{N^\alpha} \right\}$$
for $L=1,2,\ldots,K$. By definition we have
\begin{align} \label{eq1}
\left| \frac{1}{2LK} \mathcal{H}_L - N^{2-\alpha} \right| \leq F(K^2,N).
\end{align}
Following the lines of the proof of Theorem~1 in \cite{GL17} almost verbatim yields the two inequalities
\begin{align} \label{eq2}
\begin{split}
  \frac{1}{2LKN^{2-\alpha}} \mathcal{H}_L & \geq \sum_{i=0}^{\lfloor N/K \rfloor} (A_i(A_i-1)+2A_i(A_{i+1}+\ldots+A_{i+L-1}))\\
  & \geq \frac{2}{K+1} Z_K - \frac{1}{2LK}N^{\alpha-1},
 \end{split}
\end{align}
where
\begin{align} \label{eq3}
    Z_K & \geq \frac{N^{\alpha-1}}{2K^2N} \left( \frac{K^2(NB+H)^2}{K+\lfloor N^\alpha/K \rfloor} + \frac{K^2(N(1-B)-H)^2}{\lfloor N^\alpha/K\rfloor - K - \lfloor N^\alpha B/K \rfloor} \right).
\end{align}
As both denominators are positive, this implies
$$ Z_K \geq \frac{1}{2N} \left( \frac{K(NB+H)^2}{N+K^2N^{1-\alpha}} + \frac{K(N(1-B)-H)^2}{N(1-B)-K^2N^{1-\alpha}} \right).$$
The expression $Z_K$ is monotonic decreasing in the range of $K$ and furthermore we have $K^2N^{1-\alpha} \leq H/5$. Hence
\begin{align} \label{eq4}
    \begin{split}
    Z_K & \geq \frac{K}{2N} \left( \frac{(NB+H)^2}{N+H/5} + \frac{(N(1-B)-H)^2}{N(1-B)-H/5} \right)\\
    & \geq \frac{K}{2N} \left( N + \frac{16}{25}\frac{NH^2}{(NB+H/5)(N(1-B)-H/5)}\right)\\
    & > \frac{K}{2} \left( 1 + \frac{H^2}{2N^2} \right)
    \end{split}
\end{align}
From the Inequalities \eqref{eq1}--\eqref{eq4} we deduce
$$ \frac{1}{N^{2-\alpha}} F(K,N^2) + 1 \geq \max_{L=1,\ldots,K}  \frac{1}{2KLN^{2-\alpha}} \mathcal{H}_L > 1 + \frac{H^2}{2N^2} - \frac{3}{2}\frac{1}{K} - \frac{1}{2K}N^{\alpha-1}.$$
%By definition of $K$ we have $N^{2}/K \leq N^\alpha F(K^2,N)$ or $K \geq  \frac{1}{2}N^{\tfrac{2}{5}\alpha}$ (if $K < N^\alpha/F(K^2,N)$ ). This implies
Then $K\geq\tfrac{1}{2}N^{2/5\alpha}$ implies $N^2/K \leq 2N^{2(1-\tfrac{1}{5}\alpha)}$and hence
\begin{align*}
    H^2 & < \frac{4N^{2}}{K} + 2N^\alpha F(K^2,N)\\
& < 6 \max \left( \frac{N^{2}}{K}, N^\alpha F(K^2,N) \right)\\
& < 12 \max \left( N^{2(1-\tfrac{1}{5}\alpha)}, N^\alpha F(K^2,N)\right) < H^2,
\end{align*}
and thus a contradiction.
\end{proof}

Corollary~\ref{cor:alpha-pair} can be deduced from Theorem~\ref{thm:paircor_discrepancy} by a relatively short proof, which again only contains small amendments of the ideas in~\cite{GL17}. We only include the proof for the sake of completeness.
\begin{proof}[Proof of Corollary~\ref{cor:alpha-pair}] Let $\varepsilon > 0$ be arbitrary. Since $(x_n)_{n \in \mathbb{N}}$ has $\alpha$-pair correlations for $\alpha$, it follows that 
    $$\left| \frac{1}{2s} \# \left\{ 1 \leq l \neq m \leq N : \norm{x_m-x_n} \leq \frac{s}{N^\alpha}\right\} - N^{2-\alpha}  \right| \leq \varepsilon N^{2-\alpha}.$$
for $N \geq N_0$. Define $F(K,N):= \varepsilon N^{2-\alpha}$ and choose $K$ such that $\tfrac{1}{2}N^{\tfrac{2}{5}\alpha} \leq K \leq N^{\tfrac{2}{5}\alpha}$. %Then %and $N_0 \geq 1/\varepsilon^r$ for some $r > 5$ with $\alpha \geq \frac{2r-1}{2r-2}$. Let $K:=\lfloor 1 / \varepsilon^{2\alpha}\rfloor$. Thus,
%$$\frac{N^\alpha}{F(K^2,N)} = \frac{N^\alpha}{N^{2-\alpha} \varepsilon} \leq \frac{1}{N^{2-2\alpha}\varepsilon} \leq \frac{1}{\varepsilon^{-r(2-2\alpha)-1}} \leq \frac{1}{\varepsilon^{2\alpha}} \leq K \leq N^{\tfrac{2\alpha}{5}}.$$
Then Theorem~\ref{thm:paircor_discrepancy} implies
$$D_N^* \leq \frac{5}{N} \cdot \max \left( N^{1-\tfrac{1}{5}\alpha}, \sqrt{N^\alpha N^{2-\alpha} \varepsilon} \right) = 5\sqrt{\varepsilon}.$$
\end{proof}
\section{Gaps and Discrepancy} \label{sec:Gaps_and_Discrepancy}

In this section we analyze how the gap structure of a sequence influences its star-discrepancy. For a finite sequence $(x_n)_{n=1}^N \in \mathbb{T}^1$ let the gap lengths be denoted by $L_1^{(N)}<\cdots<L_K^{(N)}$ and let $N_1^{(N)},\ldots,N_k^{(N)}$ be their multiplicities. We assume that $(x_n)_{n \in \mathbb{N}}$ has the finite gap property, i.e. $K(N) \leq K$ for some $K \in \mathbb{N}$ and all $N \in \mathbb{N}$. We avoid the upper index $(N)$ in the following because $N$ is always clear from the context. It follows that
$$\sum_{k=1}^K N_k L_k = 1.$$
Furthermore we will denote by $(x_n^*)_{n=1}^N \in [0,1)$ the elements of $x_n$ ordered by magnitude, i.e. $x_1^* \leq x_2^* \leq \ldots \leq x_N^*$. In the case of Kronecker sequences, note that $x_1^* = L_j$ for some $j \leq K$ does not necessarily hold because the three gap theorem is in the version presented here a statement about $\mathbb{T}^1$ and not about $[0,1)$.
%Note that $L_{K+1} \leq \max_{i \leq K} L_i$. 
\begin{example} \label{exa:kronecker} Let $z \in \mathbb{R} \setminus \mathbb{Q}$ have bounded partial quotients, i.e. $z=[a_0,a_1,\ldots]$ with $a_i \leq R$ for all $i \in \mathbb{N}$. Consider the Kronecker sequence $(nz)_{n=1}^N$ and let $N = q_i$ be the denominator of a convergent. Then there exist only two different gap lengths by the three gap Theorem~\ref{thm:3gap}. From the basic theory of continued fractions we get 
$$ \frac{1}{(R+2)N} < \frac{1}{q_i + q_{i+1}} \leq L_j \leq \frac{1}{q_{i-1}} \leq \frac{(R+1)}{N}$$
for all $L_j$ in this case. If $q_{i-1} < N < q_{i}$, then there are at most three different gap lengths. The smallest gap length is the same as for $N=q_i$ and the maximal length is bounded from above by $\frac{1}{q_{i-1}}$ in this case, too. Hence
$$\frac{1}{(R+2)N} < \frac{1}{q_i + q_{i+1}} \leq L_j \leq \frac{1}{q_{i-1}} < \frac{(R+1)}{N}.$$
Precise upper bounds for the quantities $NL_K$ are only known in special cases and a current research topic on their own, see \cite{Moc20}. 
\end{example}
%Furthermore let us assume that $L_i^{(N)} \leq C N$ for all $i \leq K$ and all $N \in \mathbb{N}$ for some $C \in \mathbb{R}$ that does not depend on $N$.\\[12pt]
The next aim is to establish the general link between the finite gap property of a sequence and its star-discrepancy which is described in Theorem~\ref{thm:gap_discrepancy}. For that purpose it proves useful to use the following formula for calculating the star-discrepancy, see e.g. \cite{Nie92}, Theorem 2.6.
\begin{lemma} \label{lem:sd} The star-discrepancy can be calculated by the formula
$$D_N^*(x_1^*,\ldots,x_N^*) = \frac{1}{2N} + \max_{1 \leq n \leq N} \left| x_n^* - \frac{2n-1}{2N} \right|.$$
\end{lemma}
From Lemma~\ref{lem:sd} we see that $(x_n)_{n \in \mathbb{N}}$ and $(1-x_n)_{n \in \mathbb{N}}$ have the same star-discrepancy for every $N \in \mathbb{N}$.
%\begin{example}{(Continuation of Example~\ref{exa:kronecker})} In the spirit of Lemma~\ref{lem:sd}, we can deduce for $n=N$ the inequality
%$$1 - x_N^* < \max L_j < \frac{2(R+2)^2}{2N}$$
%yielding
%$$\left| x_N^* - \frac{2N-1}{2N} \right| < \frac{2(R+2)^2-1}{2N}.$$
%Similarly we obtain
%$$\left| x_1^* - \frac{1}{2N} \right| < \frac{2(R+2)^2-1}{2N}.$$
%Since $R$ is fixed for all $N \in \mathbb{N}$, this may be interpreted as an indication that $x_N$ (as is well-known) might be a sequence with a small (low) discrepancy although the corresponding difference, of course, needs to be uniformly bounded for all $x_i^*$. 
%\end{example}
This leads us to the proof of Theorem~\ref{thm:gap_discrepancy}. The notational convention for the length $L_K$ therein is, of course, a reminiscence to Kronecker sequences.

%If we assume that $L_i^{(N)} \leq C N$ for all $i \leq K$ and all $N \in \mathbb{N}$ with some constant $C \in \mathbb{R}$ that does not depend on $N$ and that gaps of each size distribute in a predictable way across the unit interval, it is possible to estimate the star-discrepancy.
%Let $(x_n)_{n \in \mathbb{N}}$ be a sequence. For $N \in \mathbb{N}$ let $x_1^*,\ldots,x_N^*$ be the elements $x_1,\ldots,x_N$ ordered by magnitude. 
%\begin{theorem} %\label{thm:gap_discrepancy} %Assume that for all $N \in N$ and $k=1,\ldots,K$ and $j=1,\ldots,J$ and 
%Assume that $(x_n)_{n \in \mathbb{N}}$ has the finite gap property and that for the largest gap length $L_K = \frac{R+2}{N}$ for some $R > - 1$ holds. If for all $j=1,\ldots,N$ we have 
%$$x_j^* = x_1^* + \sum_{i=1}^K n_k(j) L_k$$
%with $\sum_{k=1}^K n_k(j) = j-1$ and $n_k(j) \in [ \tfrac{N_k}{N}j - \varepsilon; \tfrac{N_k}{N}j + \varepsilon]$, then
%$$D_N^*(x_1,\ldots,x_N) \leq \frac{R+3}{N} + \varepsilon \sum_{k=1}^K L_k.$$
%\end{theorem}
\begin{proof}[Proof of Theorem~\ref{thm:gap_discrepancy}] At first, we bound $D_N^*(x_1,\ldots,x_N)$ from above. Since
\begin{align*}
x_j^* & = x_1^* + \sum_{k=1}^K n_k(j) L_k  \leq x_1^* + \frac{j}{N} \sum_{k=1}^K N_k L_k + \varepsilon \sum_{k=1}^K L_k\\
& = x_1^* + \frac{j}{N} + \varepsilon \sum_{k=1}^K L_k.
\end{align*}
it follows that
\begin{align*}
    x_j^* - \frac{2j-1}{2N} & \leq x_1^* + \frac{j}{N} - \frac{2j-1}{2N} + \varepsilon \sum_{k=1}^K L_k \leq \frac{1}{2N} + L_K + \varepsilon \sum_{k=1}^K L_k\\
    & \leq \frac{R+3}{N} + \varepsilon \sum_{k=1}^K L_k.
\end{align*}
In the same way, the lower bound for $n_k(j)$ can be applied to  bound $x_j^*-\frac{2j-1}{2N}$ from below by $\varepsilon \sum_{k=1}^K L_k$.
\end{proof}
In our next example we show how Theorem~\ref{thm:gap_discrepancy} can be applied to Kronecker sequences to prove that they are low-discrepancy sequences.
\begin{example}{(Continuation of Example~\ref{exa:kronecker})} The proof is split into four steps. In the first preparatory step we fix the notation and discuss Three Gap Theorem~\ref{thm:3gap} in more detail. Next we consider the simplest case, namely $N=q_i$. Even stronger bounds for the star-discrepancy can be derived from the gap structure for these $N$. Third we come to the $N$ where only two different gap lengths occur. In a final step, we consider the general case. While, of course, the usual proof of the low-discrepancy property of Kronecker sequences is much shorter, see e.g. \cite{Nie92}, our approach has the advantage that it sheds light upon how it can be derived from the gap structure.\\[12pt]
(1) Description of notation\\[12pt]
Let $q_i$ be the denominators of the convergents. According to Theorem~\ref{thm:3gap}, the number of small gaps is $S(N) = N - q_i$, and the sum of the number of medium and large gaps satisfies $M(N) + L(N) = q_i$. While the notion of small gaps remains the same for $q_i < N \leq q_{i+1}$, the notion for medium and large size changes whenever $N$ reaches $cq_i + q_{i-1}$ for some $c \in \mathbb{N}$. More precisely, if $cq_i + q_{i-1} \leq N < (c+1) q_i + q_{i-1}$, then the number of medium size gaps is $M(N) = N - cq_i - q_{i-1}$. Let $S(N,k)$ for $k \leq N$ denote the number of small gaps between $x_{1}^*$ and $x_{k+1}^*$. Accordingly $L(N,k)$ is the number of large gaps between $x_{1}^*$ and $x_{k+1}^*$ and $M(N,k)$ the corresponding number of medium gaps. In the following we (almost) exclusively pay attention to $S(N,k)$ because very similar arguments apply for $L(N,j)$ and $M(N,k)=k-S(N,k)-L(N,k)$.\\[12pt]
(2) The case $N = q_i$\\[12pt]
At first, we consider the case $N=q_i$. Therefore, there are only $q_i-q_{i-1}$ small gaps and $q_{i-1}$ large gaps. We claim that 
$$S(N,k) \in \left(\frac{S(N)}{N} k - 1,\frac{S(N)}{N} k + 1\right).$$
We prove this by induction on $i$. For $N=q_1=1$ the claim is trivial. Let $x_{1,N}^* < x_{2,N}^* \ldots < x_{N,N}^*$ be the elements $x_1, x_2, \ldots, x_{N}$ ordered by size.  We may without loss of generality assume that $x_{1,N}^* = x_{1,q_{i+1}}^*$ because either $x_{1,N}^* = x_{1,q_{i+1}}^*$ or $x_{N,N}^* = x_{q_{i+1},q_{i+1}}^*$ and we can replace $x_1,\ldots,x_N$ by $1-x_1,\ldots,1-x_N$ if necessary. When passing from $N=q_i$ to $N^*=q_{i+1}$ an (old) small gap is split up into $a_{i+1}-1$ (new) small gaps and $1$ (new) large gap as can be seen from the dynamics behind the continued fraction algorithm, compare \cite{Wei20}. Similarly, an (old) large gap is split up into $a_{i+1}$ (new) small gaps and $1$ (new) large gap. Therefore, the point $x_{k+1,N}^*$ is the same point as $x_{k^*,N^*}^*$, where
$$k^* = a_{i+1}S(N,k) + (k-S(N,k))(a_{i+1}+1) = (a_{i+1}+1)k-S(N,k).$$
From that we can calculate the number of small gaps as
$$S(N^*,k^*) = (a_{i+1}-1)S(N,k)+(k-S(N,k))a_{i+1} = ka_{i+1}-S(N,k) = k^*-k.$$
We now compare $S(N^*,k^*)$ to $\frac{S(N^*)}{N^*}k^*$ and obtain
\begin{align*}
    \left|S(N^*,k^*) \right. & \left. - \frac{S(N^*)}{N^*}k^* \right| = \left| k^*-k - \frac{S(N^*)}{N^*}k^* \right| = \left| k^* \frac{N^*-S(N^*)}{N^*} - k\right|\\
    & = \left| k^* \frac{q_i}{q_{i+1}} - k\right| = \left| ((a_{i+1}+1)k-S(N,k))\frac{q_i}{a_{i+1}q_i+q_{i-1}}-k \right|\\
    & = \left| \frac{k(q_i-q_{i-1})-S(N,k)q_i}{a_{i+1}q_i+q_{i-1}} \right| < \frac{q_i}{a_{i+1}q_i+q_{i-1}} < 1,
\end{align*}
where we applied the induction hypothesis in the penultimate step. If $k^* < j < (k+1)^*$ and if the gap between $k$ and $k+1$ was small (old), then the biggest difference between $S(N^*,j)$ and $\frac{S(N^*)}{N}j$ occurs for $j-k^*=a_{i+1}-1$. An analogous calculation as above yields
\begin{align*}
    \left|S(N^*,j) \right. & \left. - \frac{S(N^*)}{N^*}j \right| = \left|S(N^*,k^*) + a_{i+1}-1 - \frac{S(N^*)}{N^*}(k^*+a_{i+1}-1) \right|\\
    & < \frac{q_i}{a_{i+1}q_i+q_{i-1}} + \frac{(a_{i+1}-1)q_i}{a_{i+1}q_i+q_{i-1}} < 1.
\end{align*}
If $k^* < j < (k+1)^*$ and if the gap between $k$ and $k+1$ was large (old), then this gap splits into $a_{i+1}$ small gaps (new) and $1$ large gap (new). Hence it remains to consider the case $j-k^*=a_{i+1}$. However, it follows inductively as well that for large gaps we have $\left|S(N^*,k^*) - \frac{S(N^*)}{N}k^* \right| <  \frac{q_{i-1}}{a_{i+1}q_i+q_{i-1}}$ and the claim follows.\\[12pt]
(3) The case $N^* = (c+1)q_i + q_{i-1}$\\[12pt]
From (2) we now have knowledge about $S(N,k)$ for the case $N=q_i$. In this part of the proof, we use this information to obtain a corresponding result for $N^* = (c+1)q_i + q_{i-1}$: if $N=q_i$ and $N^*= q_i + q_{i-1}$ we have only two sizes of gaps in both cases and $k \mapsto k^* = k + (k- S(N,k)) = 2k - S(N,k)$ and $M(N,^*,k^*) = S(N,k) + (k - S(N,k)) = k^*-k-S(N,k)$. If $N=cq_{i}+q_{i-1}$ with $c \leq a_{i+1}-1$ and $N^* = (c+1)q_{i}+q_{i-1}$, then again $k^* =  2k - S(N,k)$ and this time $S(N^*,k^*) = k^*  - k - S(N,k)$.  A similar calculation as for $N^*=q_{i+1}$ yields $M(N,^*,k^*)) \in \left( \tfrac{M(N^*)}{N^*} k^*-1, \tfrac{M(N^*)}{N^*}k^* +1 \right)$ (and thus a corresponding statement for $S(N^*,k^*)$) in all mentioned cases.\\[12pt]
(4) All other $N$\\[12pt]
The remaining cases are those, where $N^* = N + n$ with $$N \in \left\{ q_i, q_i+q_{i-1}, 2q_i + q_{i-1}, \ldots ,(a_i-1)q_i+q_{i-1} \right\}.$$ Again the strategy for (4) is to use information for $N$ from (3) to derive results for $N^*$. Now we take a $x_{k+1}^* \in (x_i^*)_{i=1}^{N}$ and let $k^*$ denote the index of $x_k^*$ in $(x_i^*)_{i=1}^{N^*}$. If we assume for a moment $k^* = k + k \frac{n}{N} + \varepsilon$ for some $\varepsilon \in \mathbb{R}$, then $S(N^*,k^*) = S(N,k) + \frac{n}{N}k + \varepsilon$ because the new points also contribute new small gaps. This yields
\begin{align*}
    \left|S(N^*,k^*) \right. & \left. - \frac{S(N^*)}{N^*}k^*\right| = \left|S(N,k) + \frac{n}{N}k + \varepsilon -  \frac{S(N^*)}{N^*}k^* \right|\\ & \leq \left| \frac{S(N)}{N}k + \frac{n}{N} k + \varepsilon + 1 -  \frac{S(N^*)}{N^*}k^* \right|\\
    & = \left| \frac{S(N) + n}{N}k + \varepsilon + 1 - \frac{S(N^*)}{N^*} \frac{N+n}{N} k  - \frac{S(N^*)}{N^*} \varepsilon \right|\\ & = \left| \frac{S(N) + n}{N}k - \frac{S(N)+n}{N+n} \cdot \frac{N+n}{N} k + \varepsilon - \frac{S(N^*)}{N^*} \varepsilon + 1\right| <  |\varepsilon| + 1.
\end{align*}
It remains to be shown that $k^* = k + k \frac{n}{N} + \varepsilon(k)$ and to find a more explicit expression for $\varepsilon(k)$. For that purpose we split the full sequence $(x_n)_{n=1}^{N^*}$ into several subsequences. For $(x_n)_{n=1}^{N}$, we can deduce from the fact that there are only two gap lengths and the formulae for $S(N)$ and $L(N)$ that $x_k^* = x_1^* + k/N + \varepsilon_k /N$ with $\varepsilon_k < 1$. Next we consider the subsequence $(y_h)_{h=1}^{q_{i-1}} = (x_{N+h})_{h=1}^{q_{i-1}}$. As this is a shifted usual Kronecker sequence it follows that $y_h^* = y_1^* + h/q_{i-1} + \varepsilon_h \frac{1}{q_{i-1}}$ with $\varepsilon_h < 1$ for $h=1,\ldots,q_{i-1}$. Thus, both $(x_n-x_1^*)$ and $(y_h-y_1^*)$ consist up to a small error of equidistant points. Combining these results we see that between $k \cdot q_{i-1}/N  - 1$ and $k \cdot q_{i-1}/N  + 1$ of the $y_h$ are smaller than $x_{k+1}^*$. These (additional) points contribute to $k^*$.  Using the Ostrowski expansion of $N^*$, i.e. writing
$$N^*  = \sum_{n=1}^{l(N^*)} b_n q_n$$
with $0 \leq b_n \leq a_n$ and $b_{n-1} = 0$ if $b_n = a_n$, see e.g. \cite{Nie92}, and induction, it follows that the total number of elements of $(z_n) = (x_n)_{n=N+1}^{N^*}$ which are smaller than $x_{k+1}^*$ deviates from $kn/N$ by at most $\sum_{l \leq i} a_l$. Hence $$k^* \in \left( j + j\frac{n}{N} - \sum_{l \leq i} a_l, j + j \frac{n}{N} + \sum_{l \leq i} a_l \right).$$
Finally, we have to consider the new points $x_h^*$ which were not already in the sequence $(x_n)_{n=1}^{N}$. These new points are those that split an (old) large gap into a small and a medium one. Hence they can contribute at most $1$ to $S(N,k)$. Since $\sum_{l \leq i} a_l \leq c_0\log(N)$, we get in total that
$$S(N,k) \in \left( \frac{S(N)}{N}k - c \log(N), \frac{S(N)}{N}k + c \log(N) \right).$$
This finishes the proof that Kronecker sequences with bounded continued fraction expansion are low-discrepancy sequences.
\end{example}
Eventually, we want to apply Theorem~\ref{thm:gap_discrepancy} to van der Corput sequences to show their low-discrepancy property. They are defined as follows: for an integer $b \geq 2$ the $b$-ary representation of $r \in \NN$ is $r = \sum_{j=0}^\infty a_j(r) b^j$ with $a_j(r) \in \NN$. The radical-inverse function is defined by $g_b(r)=\sum_{j=0}^\infty a_j(r) b^{-j-1}$ for all $n \in \NN$. Finally, the van der Corput sequence in base $b$ is given by $(x_r) = g_b(r)$. In order to check the assumptions of Theorem~\ref{thm:gap_discrepancy} for van der Corput sequences, we use the following intermediate result of the proof of Theorem 3.6 in \cite{Nie92}.
\begin{lemma} \label{lem:niederreiter} Let $(x_r)_{r=1}^{N}$ be the van der Corput sequence in base $b$ and $N \leq b^{e}-1$, then the number of points $N(k)$ in $[0,x_k]$ satisfies
$$|Nx_k -N(k)| \leq \frac{1}{2}e(b-1) + 1.$$
\end{lemma}
\begin{example} \label{exa:vdc} Consider the van der Corput sequence in base $b$. We add $0$ as zeroth element of the sequence. If $N=b^e-1$, then there is only one gap length and showing the assumptions of Theorem~\ref{thm:gap_discrepancy} is trivial. If $N=ab^e-1$ with $a < b$, then the sequence consists of $b^e$ blocks with $(a-1)$ short and $1$ long gaps. The proof is again trivial with $\varepsilon < (a-1) < b$. Now let $ab^e - 1 \leq N < (a+1)b^{e} - 1$ and consider $(x_r)_{r=1}^N$ and its ordered version $(x_r^*)$. The largest gap length is at most $\frac{b}{N}$. Write $R$ for the real number such that $L_K=(R+2)/N$. Now we set $N := ab^e$ and $N^* := N + n$ with $n < b^e$. Similarly as for Kronecker sequences, we split up the complete sequence into the subsequences $(z_r)_{r=0}^{ab^e-1} = (x_r)_{r=0}^{ab^e-1}$ and $(y_j)_{j=1}^n = (x_{ab^e-1+j})_{j=1}^n$.\\[12pt] 
Any point $z_k^*$  must be of the form $z_k^* =  \frac{ib+l}{b^{e+1}}$ with $0 \leq i \leq b^{e-1}-1$ and $0 \leq l \leq a-1$ where $k=i \cdot a + l$. We know that there are $S(N,k) \in \left( \frac{(a-1)b^e}{ab^e}k -1, \frac{(a-1)b^e}{ab^e}k + 1 \right)$ small gaps between $z_{0}^*$ and $z_k^*$ and that $L(N,k) = k - S(N,k)$. Note that any point $y_j^*$ cuts a long interval of the sequence $(z_r^*)$ into a small one and a medium one. According to Lemma~\ref{lem:niederreiter}, the interval $[0,z_k^*)$ contains in total 
\begin{align*}
k^* & := N^*z_k^* + \varepsilon_k = (ab^e+n)\frac{ib+l}{b^{e+1}} + \varepsilon_k = k + \frac{a-b}{b}l + nz_k^* + \varepsilon_k
\end{align*}
points of $(x_r)$ with $\varepsilon_k \leq \tfrac{1}{2} (b-1)e + 1$. Comparing this number to $\frac{N^*}{N}k$ yields
\begin{align*}
    \left| \frac{N^*}{N}k - k^* \right| &= \left| \frac{n}{ab^e}k - \frac{a-b}{b}l - n\frac{ib+l}{b^{e+1}} + \varepsilon_k \right|\\ 
    & = \left| \frac{1}{b}(b-a) l \left( \frac{n}{ab^e} + 1\right) + \varepsilon_k \right| < \left| \frac{1}{b}(b-a)(a-1)\left(\frac{1}{a}+1\right) + \varepsilon_k \right|\\
    & < \frac{3}{8}(b-1) + \frac{1}{2}e(b-1) + 1 =:\varepsilon
\end{align*}
Now let $S(N^*)$ be the number of small gaps of $(x_r^*)$ and let $S(N^*,k)$ be the number of small gaps between $x_0^*$ and $x_k^*$. Note that $S(N^*) = S(N) + n$ and $S(N^*,k^*) = S(N,k)+k^*-k$. Hence
\begin{align*}
    \left| \frac{S(N^*)}{N^*}k^* \right. & - \left. S(N^*,k^*)\right| = \left| \frac{S(N^*)}{N}k + \varepsilon \frac{S(N^*)}{N^*} - \left(S(N,k)+k^*-k\right) \right| \\
    & \leq \left| \frac{S(N)+n}{N}k + \varepsilon \frac{S(N)+n}{N+n}  - \left( \frac{S(N)}{N}k + \tilde{\varepsilon} + \frac{N^*}{N} k + \varepsilon - k \right) \right| \\
    & = \varepsilon \frac{N-S(N)}{N^*} + \tilde{\varepsilon} \leq \frac{3}{8}(b-1) + \frac{1}{2}e(b-1) + 1 + 1\\
    & < \frac{1}{2}(e+1)(b-1) + 2 < c \frac{\log N}{\log b}.
\end{align*}
A corresponding result for $L(N^*,k)$ can be derived in the same manner and $M(N^*,k) = k - L(N^*,k) - S(N^*,k)$ follows. Thus, van der Corput sequences satisfy the assumptions of Theorem~\ref{thm:gap_discrepancy} with $\varepsilon = \varepsilon(N) = c \frac{\log N}{\log b}$. We obtain the well-known result that van der Corput sequences are low-discrepancy sequences.
\end{example}
\section*{Acknowledgments}
Parts of the research on this paper was conducted during a stay of the author at the Max-Planck Institute in Bonn whom he would like to thank for hospitality and an inspiring scientific atmosphere. Moreover, the author would like to thank the referee for many useful comments which helped to improve the presentation.

%Hence, if $N$ is big enough every ball of radius $sN^{-\alpha}$ contains between $N\vol(B(0,sN^{-\alpha}) - 1$ and $N\vol(B(0,sN^{-\alpha}) + 1$ points because $\vol(B(0,sN^{-\alpha}) = o(N^{-\alpha \cdot d})$ and $\alpha \cdot d < 1 - \varepsilon$. Note that for $\alpha = \tfrac{1}{d}$, the number of points might be zero. This it the reason why low-discrepancy sequences can fail to have Poissonian pair correlations. 
%Similarly,
%\begin{align*}
%	F_N^\alpha(s) = \frac{1}{N^{2}} & \frac{\# \left\{ 1 \leq l \neq m \leq N \ : \ \norm{x_l - x_m} \leq \frac{s}{N^{\alpha}} \right\}}{\vol(B(0,sN^{-\alpha})}\\
%	& \geq \frac{1}{N^2} \cdot N \cdot \frac{N\vol(B(0,sN^{-\alpha}) - 1}{\vol(B(0,sN^{-\alpha})}\\
%	& = 1 - \frac{1}{N\vol(B(0,sN^{-\alpha})}
%\end{align*}

\bibliographystyle{acm}
\bibdata{references}
\bibliography{references}
\end{document}